\documentclass[graybox]{svmult}

\usepackage{mathptmx}       
\usepackage{helvet}         
\usepackage{courier}        
\usepackage{type1cm}        
\usepackage{makeidx}         
\usepackage{graphicx}        
\usepackage{multicol}        
\usepackage[bottom]{footmisc}
\usepackage{amsfonts, amssymb, amsmath}

\makeindex             

\begin{document}
\titlerunning{Periodic generalized function algebras}
\title*{Generalized Function Algebras Containing  Spaces of Periodic Ultradistributions}
\author{Andreas Debrouwere}
\institute{Andreas Debrouwere,  Department of Mathematics, Ghent University, Krijgslaan 281 Gebouw S22, 9000 Gent, Belgium, 
\email{Andreas.Debrouwere@Ugent.be} \newline Supported by Ghent University, through a BOF Ph.D.-grant.}
\maketitle
\vspace{-2cm}\abstract{We construct differential algebras in which spaces of (one-dimensional) periodic ultradistributions are embedded. By proving a Schwartz impossibility type result, we show that our embeddings are optimal in the sense of being consistent with the pointwise multiplication of ordinary functions. In particular, we embed the space of hyperfunctions on the unit circle into a differential algebra in such a way that the multiplication of real analytic functions on the unit circle coincides with their pointwise multiplication. Furthermore, we introduce a notion of regularity in our newly defined algebras and show that  an embedded ultradistribution is regular if and only if it is an ultradifferentiable function.}

\section{Introduction}
Differential algebras containing the space of distributions were first introduced and studied by J.F. Colombeau \cite{Colombeau84, Colombeau85}. It was the starting point of the by now well-established nonlinear theory of generalized functions. For a clear exposition of this theory and its applications to various branches of analysis we refer to the monographs \cite{GGK, Ober}. 

In this paper we contribute to this program by developing a nonlinear theory of (one-dimensional) periodic ultradistributions. The problem of embedding the space of hyperfunctions on the unit circle \cite{Morimoto} into a differential algebra has attracted the attention of various authors \cite{Delcroix2,Delcroix,Valmorin}. However, the embeddings proposed so far do not preserve the multiplication of all real analytic functions. In this article we construct a new algebra, containing the space of hyperfunctions on the unit circle, which does enjoy this property and we show that it is optimal in this respect. In fact, we consider the aforementioned embedding problem for general classes of periodic ultradistributions, both of Beurling and Roumieu type, defined via a weight sequence $M_p$ on which only very mild growth conditions are imposed. The Roumieu case of $M_p= p!$ corresponds with the hyperfunction case. 

In \cite{Delcroix2} (see also \cite{Delcroix}) a differential algebra containing the space of hyperfunctions on the unit circle (the topological dual of the space of real analytic functions on the unit circle) is constructed and it is shown that the multiplication of functions belonging to the quasianalytic class defined via the weight sequence $p!^s$, for some $0 < s <1$, is preserved. A comparison with the situation for classical distributions (the topological dual of the space of compactly supported smooth functions) shows that this is not the most optimal result one would presume to be true. Namely, recall that Schwartz' impossibility result \cite{Schwartz}  asserts there is no associative and commutative differential algebra, containing the space of distributions, in which the multiplication of $k$-times differentiable functions, $k \in \mathbb{N}$, coincides with their pointwise multiplication. Nonetheless, in the Colombeau algebra the multiplication of smooth functions is preserved. In analogy to Colombeau's construction, it is natural to expect that it is possible to embed the space of hyperfunctions on the unit circle into a differential algebra in such a way that the multiplication of real analytic functions is preserved and, moreover, that such an embedding is optimal. In this paper, we construct such an algebra and show its optimality. 

Stated in terms of a general weight sequence $M_p$, our results may be summarized as follows: We show that it is possible to embed the space of periodic ultradistributions of class $M_p$ into a differential algebra in such a way that the multiplication of periodic ultradifferentiable functions of class $M_p$ is preserved. By establishing an analogue of Schwartz' impossibility result for periodic ultradistributions, we show that our embedding is optimal.
Furthermore, we introduce a notion of regularity in our algebras of periodic generalized functions and show that an embedded ultradistribution is regular if and only if it is an ultradifferentiable function.

Finally, we would like to remark that in a forthcoming paper \cite{DVV}, jointly with H. Vernaeve and J. Vindas, a nonlinear theory for (non-quasianalytic) ultradistributions on the $n$-dimensional Euclidean space is developed (see also the earlier works \cite{Benmeriem, Gramchev, PilScar}). There we construct a sheaf of differential algrebras  in which the sheaf of ultradistributions of class $M_p$ is embedded and the multiplication of ultradifferentiable functions of class $M_p$ is preserved. Moreover, we define a notion of regularity and lay the ground for microlocal analysis in these algebras. The results in the present paper are the natural counterparts of these results in the periodic case.

This paper is organized as follows. In Sect. \ref{section-preliminaries} we introduce certain spaces of periodic ultradifferentiable functions and ultradistributions, and recall their characterization in terms of Fourier coefficients \cite{Gorbagcuk,Petzsche,Pilipovic1}. Section \ref{section-structuretheorem} is devoted to an analogue of Komatsu's second structure theorem for periodic ultradistributions, this result is used to prove the Schwartz impossibility type result for periodic ultradistributions in Sect. \ref{section-impossibility}. The construction of our algebras is given in Sect. \ref{section-algebra}. Furthermore, we give an alternative projective description of the algebras of Roumieu type (see \cite{Komatsu3, Pil94} for analogues in the theory of ultradistributions) and provide a null characterization of the so called space of negligible elements, which we apply to obtain a pointwise characterization of our generalized functions. We remark that the definition of our algebras containing the space of hyperfunctions on the unit circle significantly differs from the one considered in \cite{Delcroix2,Delcroix,Valmorin}. 
The embedding of spaces of periodic ultradistributions of class $M_p$ into our algebras is discussed in Sect. \ref{section-embedding}. As usual, this will be achieved by means of convolution with a suitable mollifier sequence. We also show that the multiplication of ultradifferentiable functions of class $M_p$ is preserved. It should be pointed out that in the Roumieu case this result is valid precisely because of the (different) definition of our algebras. In the last section we study regularity in our algebra of periodic generalized functions and show that an embedded ultradistribution is regular if and only if it is an ultradifferentiable function of class $M_p$. 

\section{Spaces of Periodic Ultradifferentiable Functions \newline and Their Duals}\label{section-preliminaries}
In this preliminary section we introduce test function spaces of periodic ultradifferentiable functions and recall the well known characterization of these spaces 
and their duals by means of Fourier coefficients \cite{Gorbagcuk,Petzsche,Pilipovic1}.  For the sake of completeness, we discuss these results in quite some detail and provide proofs.
We shall work with the notion of ultradifferentiability  through weight sequences 
\cite{Komatsu1}. Fix a positive weight sequence $(M_p)_{p \in \mathbb{N}}$ with $M_0 = 1$. We always assume the following conditions 
on $M_p$:
\begin{enumerate}
\item [] $(M.1)$\, $M^{2}_{p}\leq M_{p-1}M_{p+1}\;,$  $p \in \mathbb{Z}_+$\;,
\item [] $(M.2)$\, $\displaystyle M_{p+q}\leq A H^{p+q} M_{p} M_{q}$\;, $p,q\in\mathbb{N}$\;, for some $A,H \geq 1$\;.
\end{enumerate}
The \emph{associated function} of $M_p$ is defined as
$$
M(t):=\sup_{p\in\mathbb{N}}\log\frac{t^p}{M_p}\;,\qquad t > 0\;.
$$
We extend $M$ to the whole real line by setting $M(t) = M(|t|)$ for $t \in \mathbb{R}$.
\begin{lemma}\cite[Prop.\ 3.6]{Komatsu1}\label{assumptionM2}
Let $A,H$ be the constants occurring in $(M.2)$. Then
$$
2M(t) \leq M(Ht) + \log A\;, \qquad t > 0\;.
$$
\end{lemma}
As usual, the relation $M_p\subset N_p$ between two weight sequences means that there are $C>0$ and $h>0$ such 
that $M_p\leq Ch^{p}N_{p}$ for all $p\in\mathbb{N}$. The stronger relation $M_p\prec N_p$ means that the latter 
inequality remains valid for every $h>0$ and a suitable $C=C_{h}>0$.
A function $f$ defined on $\mathbb{R}$ is said to be $2\pi$-periodic if
$$
f(t + 2\pi) = f(t), \qquad t \in \mathbb{R}.
$$
  We write $\mathcal{E}_{2\pi}^{M_p,h}$, $h>0$, for the Banach space of all smooth $2\pi$-periodic functions on $\mathbb{R}$ satisfying
$$
\| \varphi\|_{M_p,h} = \| \varphi\|_{h} := \sup_{p \in \mathbb{N}} \frac{h^p\|D^p\varphi\|_{L^\infty(\mathbb{R})}}{M_p} < \infty \;,
$$
where $D = -\imag \D/\D t$. Furthermore, we define
$$
\mathcal{E}_{2\pi}^{(M_p)} = \varprojlim_{h \rightarrow \infty}\mathcal{E}_{2\pi}^{M_p,h}\;, \quad \mathcal{E}_{2\pi}^{\{M_p\}} = \varinjlim_{h \rightarrow 
0^+}\mathcal{E}_{2\pi}^{M_p,h}\;.
$$
Elements of their dual spaces $\mathcal{E}_{2\pi}'^{(M_p)}$ and $\mathcal{E}_{2\pi}'^{\{M_p\}}$ are called \emph{periodic ultradistributions of 
class $(M_p)$ (of Beurling type)} and \emph{periodic ultradistributions of class $\{M_p\}$ (of Roumieu type)}, respectively.
In the sequel we shall write $\ast$ instead of $(M_p)$ or $\{M_p\}$ if we want to treat both cases simultaneously. In addition, we shall often first state assertions for the Beurling case followed in 
parenthesis by the corresponding statements for the Roumieu case.
\begin{proposition}
The sequence $M_p$ satisfies $1 \prec M_p$ ($1 \subset M_p$) if and only if the space $\mathcal{E}_{2\pi}^{(M_p)}$ (the space $\mathcal{E}_{2\pi}^{\{M_p\}}$ ) contains a function which is not identically zero.
\end{proposition}
\begin{proof}
If $M_p$ satisfies $1 \prec M_p$ ($1 \subset M_p$) then $\E^{\imag \lambda t} \in \mathcal{E}_{2\pi}^{(M_p)}$ ($\E^{\imag \lambda t} \in \mathcal{E}_{2\pi}^{\{M_p\}}$) for each $\lambda \in \mathbb{R}$. Conversely, suppose that $\mathcal{E}_{2\pi}^{(M_p)}$ ($\mathcal{E}_{2\pi}^{\{M_p\}}$ ) contains a function $\varphi$ which is not identically zero. Then there is $\lambda \in \mathbb{R}$ such that 
$$
 \int_0^{2\pi} \varphi(t)\E^{\imag \lambda t} \: \D t \neq 0.
$$
Hence, for each $h > 0$ (some $h >0$) it holds that
$$
0 < \left |\lambda^p \int_0^{2\pi}\varphi(t) \E^{\imag \lambda t} \:\D t \right| = \left | \int_0^{2\pi} D^ p\varphi(t) \E^{\imag \lambda t} \:\D t \right| \leq \frac{2\pi \|
\varphi\|_h M_p}{h^p}\;, \qquad p  \in \mathbb{N}\;,
$$
from which the result follows.\qed
\end{proof}
We shall always assume $1 \prec M_p$ and $1 \subset M_p$ in the Beurling and Roumieu case, respectively. 
\begin{remark}
We write $\mathcal{O}_\lambda$, $\lambda > 1$, for the space of holomorphic functions on the open annulus  $\{ z \in \mathbb{C} 
\, | \, 1/ \lambda < |z| < \lambda \}$  and set 
$$
\mathcal{A}(\mathbb{T}) =  \varinjlim_{\lambda \rightarrow 1^+}\mathcal{O}_\lambda\;,
$$
where $\mathbb{T}$ denotes the complex unit circle. The spaces $\mathcal{A}(\mathbb{T})$  and $\mathcal{E}_{2\pi}^{\{p!\}}$ are isomorphic as 
topological vector spaces (in the sequel abbreviated as t.v.s.) by means of the mapping
$\varphi \to \widetilde{\varphi}(t) = \varphi(\E^{\imag t})$. Elements of the dual space $\mathcal{A}'(\mathbb{T})$ are called 
\emph{hyperfunctions on the unitc circle} \cite{Morimoto}.
\end{remark}
Let $\varphi \in \mathcal{E}_{2\pi}^{\ast}$, the \emph{Fourier coefficient of index $k$} of $\varphi$, $k \in \mathbb{Z}$, is given by
$$
\widehat{\varphi}(k) := \frac{1}{2\pi} \int_0^{2\pi} \varphi(t)\E^{-\imag kt} \: \D t\;.
$$
In order to study the Fourier coefficients of periodic ultradifferentiable functions we introduce the following sequence spaces
$$
s^{M_p,\lambda}(\mathbb{Z}) = \{ (c_k)_k \in \mathbb{C}^{\mathbb{Z}} \, | \, \sigma_\lambda( (c_k)_k) :=  \sup_{k \in \mathbb{Z}} |
c_k| \E^{M(\lambda k)} < \infty \}\;, \qquad \lambda > 0\;.
$$
Furthermore, we set
$$
s^{(M_p)}(\mathbb{Z}) = \varprojlim_{\lambda \rightarrow \infty} s^{M_p,\lambda}(\mathbb{Z})\;, \quad s^{\{M_p\}}(\mathbb{Z}) = 
\varinjlim_{\lambda \rightarrow 0^+}s^{M_p,\lambda}(\mathbb{Z})\;.
$$
\begin{proposition}\label{coefficients1}
For $\varphi \in \mathcal{E}_{2\pi}^{\ast}$ the following series expansion holds
\begin{equation}
\varphi(t)= \sum_{k \in \mathbb{Z}} \widehat{\varphi}(k)\E^{\imag kt}\;, \qquad t \in \mathbb{R}\;,
\label{series}
\end{equation}
with convergence in $\mathcal{E}_{2\pi}^{\ast}$. Moreover, the mapping $\varphi \to (\widehat{\varphi}(k))_k$ yields the (t.v.s.) 
isomorphism $\mathcal{E}_{2\pi}^{\ast} \cong s^{\ast}(\mathbb{Z})$.
\end{proposition}
\begin{proof}
Let $\varphi \in \mathcal{E}_{2\pi}^{M_p,h}$, $h>0$. For all $p \in \mathbb{N}$ we have 
$$
|k^p\widehat{\varphi}(k)| = \frac{1}{2\pi} \left | \int_0^{2\pi} D^p\varphi(t) \E^{-\imag kt} \:\D t \right| \leq \frac{\|
\varphi\|_h M_p}{h^p}\;, \qquad k  \in \mathbb{Z}\;.
$$
Hence 
$$
|\widehat{\varphi}(k)| \leq \|\varphi\|_h \inf_{p \in \mathbb{N}} \frac{M_p}{(h|k|)^p} = \|\varphi\|_h\E^{-M(hk)}\;, 
\qquad k \in \mathbb{Z}\;, 
$$
which shows the continuity of the mapping $\varphi \to (\widehat{\varphi}(k))_k$ in both the Beurling and Roumieu case. Conversely, for 
$ (c_k)_k \in s^{M_p,\lambda}(\mathbb{Z})$, $\lambda > 0$, and $N \in \mathbb{N}$ we have
\begin{align*}
\|D^p \sum_{|k| \geq N} c_k\E^{\imag kt} \|_{L^\infty(\mathbb{R})} &\leq \sigma_\lambda( (c_k)_k) \sum_{|k| \geq N} k^p \E^{-M(\lambda 
k)}  \\
&\leq \sigma_\lambda( (c_k)_k)\frac{M_{p+2}}{\lambda ^{p+2}} \sum_{|k| \geq N} \frac{1}{k^2}\;,
\end{align*}
which, by $(M.2)$, shows that the series $\sum_k c_k \E^{\imag kt}$, $(c_k)_k \in s^\ast(\mathbb{Z})$, converges in $
\mathcal{E}_{2\pi}^{\ast}$ and that the mapping 
$(c_k)_k \to \sum_{k \in \mathbb{Z}} c_k \E^{\imag kt}$ is continuous in both cases. Since the expansion \eqref{series} holds in the 
space of smooth $2\pi$-periodic functions endowed with its natural topology \cite{Morimoto} the proof is completed.\qed
\end{proof}
Next, we consider Fourier series of periodic ultradistributions. Observe first that a standard argument shows that the topological duals of the spaces 
$s^{(M_p)}(\mathbb{Z})$ and $s^{\{M_p\}}(\mathbb{Z})$ are given by
$$
s'^{(M_p)}(\mathbb{Z}) = \varinjlim_{\lambda \rightarrow \infty} s^{M_p,-\lambda}(\mathbb{Z})\;, \quad s'^{\{M_p\}}(\mathbb{Z}) = \varprojlim_{\lambda 
\rightarrow 0^+}s^{M_p,-\lambda}(\mathbb{Z})\;,
$$
where
$$
s^{M_p,-\lambda}(\mathbb{Z}) = \{ (c_k)_k \in \mathbb{C}^\mathbb{Z} \, | \, \sigma'_\lambda( (c_k)_k) :=  \sup_{k \in \mathbb{Z}} |c_k| \E^{-M(\lambda k)} < \infty \}\;, 
\qquad \lambda > 0\;.
$$
Let $f \in \mathcal{E}_{2\pi}'^{\ast}$, the \emph{Fourier coefficient of index} $k$ of $f$, $k \in \mathbb{Z}$, is given by 
$$
\widehat{f}(k) := \frac{1}{2\pi} \langle f(t),\E^{-\imag kt}\rangle\;.
$$
In view of the above remark, Prop. \ref{coefficients1} implies the following:
\begin{proposition}\label{coefficients2}
For $f \in \mathcal{E}_{2\pi}'^{\ast}$ the following series expansion holds
$$
f(t)= \sum_{k \in \mathbb{Z}} \widehat{f}(k)\E^{\imag kt}\;,
$$
with (strong) convergence in $\mathcal{E}_{2\pi}'^{\ast}$. Moreover, the mapping $f \to (\widehat{f}(k))_k$ yields the (t.v.s.) 
isomorphism $\mathcal{E}_{2\pi}'^{\ast} \cong s'^{\ast}(\mathbb{Z})$.
\end{proposition}
An entire function $P(z) = \sum_{n=0}^\infty a_n z^n$, $z\in \mathbb{C}$, is 
said to be an \emph{ultrapolynomial of class $(M_p)$ (class $\{M_p\}$)} if the coefficients satisfy the estimate
$$
|a_n| \leq \frac{CL^n}{M_n}\;, \qquad n \in \mathbb{N}\;,
$$
for some $L >0$ and $C> 0$ (for every $L>0$ and a suitable $C=C_{L}>0$). The associated infinite order differentiable operator
$$
P(D) = \sum_{n=0}^\infty a_n D^n\;,
$$
is called an \emph{ultradifferential operator of class $(M_p)$ (class $\{M_p\}$)}. 
Condition $(M.2)$ ensures that $P(D)$ acts continuously on $\mathcal{E}_{2\pi}^\ast$ \cite[Thm. 2.12]{Komatsu1} and hence it can be defined 
on the corresponding ultradistribution space by duality, namely, for $f \in \mathcal{E}_{2\pi}'^\ast$
$$
\langle P(D) f , \varphi \rangle = \langle f , P(-D) \varphi \rangle\;, \qquad \varphi  \in \mathcal{E}_{2\pi}^\ast\;. 
$$
One has the following relation between Fourier coefficients
$$
\widehat{P(D)f}(k) = P(k)\widehat{f}(k)\;, \qquad k \in \mathbb{Z}\;.
$$
Moreover,
\begin{equation}
P(D)(\E^{\imag kt} f) = \E^{\imag kt}P(D + k)f\;, \qquad f \in \mathcal{E}_{2\pi}'^\ast\;, k \in \mathbb{Z}\;,
\label{Leibniz}
\end{equation}
where $P(D + k)$ is the ultradifferential operator related to the symbol $P(z + k)$. Remark that this may be seen as a weak form of Leibniz' rule.
Finally, we briefly review the convolution of periodic ultradistributions: For $f, g \in \mathcal{E}_{2\pi}'^{\ast}$ their convolution is defined as
$$
\langle f \ast g, \varphi \rangle = \langle f(t), \langle g(u), \varphi(t+u) \rangle \rangle\;, \qquad \varphi \in \mathcal{E}_{2\pi}^\ast\;. 
$$
One can readily check that $f \ast g \in  \mathcal{E}_{2\pi}'^{\ast}$ and that $\widehat{f \ast g}(k) = 2\pi \widehat{f}(k) \widehat{g}(k)$ for all $k \in \mathbb{Z}$. If $g \in \mathcal{E}_{2\pi}^\ast$, then  $f \ast g \in \mathcal{E}_{2\pi}^\ast$ and it is given by
 $$
 f\ast g(t) = \langle f(u), g(t-u) \rangle, \qquad t \in \mathbb{R}\;.
$$ 
\section{Structure Theorem for Periodic Ultradistributions}\label{section-structuretheorem}
\label{section-structure}
We present an analogue of Komatsu's second structure theorem (as stated in \cite{Takiguchi}) for periodic 
ultradistributions. In the next section we shall use it to prove a Schwartz impossibility type result for periodic ultradistributions. In order to treat
the Beurling and Roumieu case uniformly we first give an alternative projective description of the spaces $\mathcal{E}_{2\pi}^{\{M_p\}}$ and $s^{\{M_p\}}(\mathbb{Z})$ (an 
idea which goes back to Komatsu \cite{Komatsu3}). We write $\mathcal{R}$ for the family of non-decreasing sequences $(r_j)_{j \in \mathbb{N}}$ with $r_0 = 1$
which tend to infinity. This set is partially ordered and directed by the relation $r_j \preceq s_j$, namely, there is $j_0 \in 
\mathbb{N}$ such that $r_j \leq s_j$ for all $j \geq j_0$. Let $r_j \in \mathcal{R}$. The function $M_{r_j}$ denotes the associated 
function of the weight sequence $M_p\prod_{j = 0}^pr_j$. We write $\mathcal{E}_{2\pi}^{M_p,r_j}$ for the Banach space of all smooth $2\pi$-periodic functions on $\mathbb{R}$ such that
\begin{equation}
\| \varphi\|_{M_p,r_j} = \| \varphi\|_{r_j} := \sup_{p \in \mathbb{N}} \frac{\|D^p\varphi\|_{L^\infty(\mathbb{R})}}{M_p\prod_{j = 0}^pr_j} < \infty\;.
\label{seminorms}
\end{equation}
\begin{lemma} \cite[Lemma 2.3]{Prangoski}\label{M2sequence} 
Let $r_j \in \mathcal{R}$. Then, there is $r'_j \in \mathcal{R}$ such that $r'_j \leq r_j$, $j \in \mathbb{N}$, and
$$
\prod_{j=0}^{p+q} r'_j \leq 2^{p+q} \prod_{j=0}^{p} r'_j \prod_{j=0}^{q} r'_j\;, \qquad \forall p,q \in \mathbb{N}\;.
$$
\end{lemma}
In view of Lemma \ref{M2sequence} one can show, in the same way as in Prop. \ref{coefficients1}, that the mapping $\varphi \to (\widehat{\varphi}(k))_k$ yields the (t.v.s.) isomorphism 
\begin{equation}
\varprojlim_{r_j \in \mathcal{R}} \mathcal{E}_{2\pi}^{M_p,r_j} \cong \varprojlim_{r_j \in \mathcal{R}} s^{M_p,r_j}\;, 
\label{coefficients3}
\end{equation}
where 
$$
s^{M_p,r_j} = \{ (c_k)_k \in \mathbb{C}^\mathbb{Z} \, | \, \sigma_{r_j}( (c_k)_k) :=  \sup_{k \in \mathbb{Z}} |c_k| \E^{M_{r_j}(k)} < \infty \}\;.
$$
\begin{lemma}\label{projectivelemma} Let $(a_n)_{n \in \mathbb{N}}$ be a sequence of positive reals. Then 
\begin{itemize}
\item[]$(i)$\, $\sup_{n} a_n \E^{-M(\lambda n)} < \infty$ for all $\lambda > 0$ if and only if $\sup_{n} a_n \E^{-M_{r_j}(n)} < \infty
$ for some sequence $r_j \in \mathcal{R}$. 
\item[]$(ii)$\, $\sup_{n} a_n \E^{M_{r_j}(n)} < \infty$ for all sequences $r_j \in \mathcal{R}$ if and only if $\sup_{n} a_n \E^{M(\lambda n)} < \infty$ for some $\lambda > 0$.
\end{itemize}
\end{lemma}
\begin{proof}
We only need to show the direct implications, the "if" parts are clear.

$(i)$ We first show that there is a subordinate function $\varepsilon$ (which means that $
\varepsilon$ is continuous, increasing, and satisfies $\varepsilon(0) = 0$ and $\varepsilon(t) = o(t)$)  such that $\sup_{n} 
a_n \E^{-M(\varepsilon(n))} < \infty$. Our assumption and Lemma \ref{assumptionM2} imply that for every $\lambda > 0$ there is 
$N = N_\lambda \in \mathbb{N}$ such that  $a_n \leq A \E^{M(\lambda n)}$ 
for all $n \geq N$, where  $A$ is the constant occurring in  $(M.2)$. Hence we can inductively determine a sequence $(N_j)_{j \in 
\mathbb{Z}_+}$ of positive natural numbers with $N_1 = 0$ which satisfies 
$$
a_n \leq A\E^{M(n/(j+1))}\;, \quad  n \geq N_j\;, \qquad \frac{N_{j}}{j} \geq \frac{N_{j-1}}{j-1} + 1\;, \qquad j \geq 2\;.  
$$
Let $l_j$ denote the line through the points $(N_j, N_j/j)$ and $(N_{j+1}, N_{j+1}/(j+1))$, and define
$$
\varepsilon(t) = l_j(t)\;, \qquad \mbox{for } t \in [N_j, N_{j+1})\;.
$$
The function $\varepsilon$ is subordinate and, moreover, $a_n \leq A\E^{M(\varepsilon(n))}$ for all $n \geq N_2$. Thus, it suffices to show that there is a sequence $r_j \in \mathcal{R}$ such that $M(\varepsilon(t)) \leq M_{r_j}(t) + C$ for $t > 0$ and some $C>0$. By \cite[Lemma 3.12]{Komatsu1} there is a weight sequence $N_p$ satisfying $(M.1)$ with associated function $N$ such that $M_p \prec N_p$ and $M(\varepsilon(t)) \leq N(t)$ for $t >0$. On the other hand, \cite[Lemma 3.4]{Komatsu3} yields the existence of a sequence $r_j \in \mathcal{R}$ with  $M_p\prod_{j=0}^p r_j \prec N_p$. The result now follows from  \cite[Lemma 3.10]{Komatsu1}.

$(ii)$ By the last part of the proof of $(i)$ and our assumption we have that $\sup_{n} a_n \E^{M(\varepsilon(n))} < \infty$ for all subordinate functions $\varepsilon$. Suppose that $\sup_{n} a_n \E^{M(\lambda n)} < \infty$  does not hold for any $\lambda > 0$. We could therefore find a sequence $(N_j)_{j \in \mathbb{Z}_+}$ of positive natural numbers with $N_1 = 0$ which satisfies 
$$ 
a_{N_j} \E^{M(N_j/j)}\geq j\;, \qquad \frac{N_{j}}{j} \geq \frac{N_{j-1}}{j-1} + 1\;, \qquad j \geq 2\;.  
$$
Exactly as in the proof of $(i)$, we define the subordinate function $\varepsilon(t)$ as the line through $(N_j, N_j/j)$ and $(N_{j+1}, N_{j+1}/(j+1))$ for $t \in [N_j, N_{j+1})$. We would then have
$$
a_{N_j}\E^{M(\varepsilon(N_j))} = a_{N_j}\E^{M(N_j/j)} \geq j\;, \qquad j \geq N_2\;,
$$
contradicting $\sup_{n} a_n \E^{M(\varepsilon(n))} < \infty$. \qed
\end{proof}
\begin{corollary} \label{projectiveiso}
We have
$$
\mathcal{E}_{2\pi}^{\{M_p\}} = \varprojlim_{r_j \in \mathcal{R}} \mathcal{E}_{2\pi}^{M_p,r_j}\;, \quad s^{\{M_p\}} (\mathbb{Z}) = \varprojlim_{r_j \in \mathcal{R}} s^{M_p,r_j} (\mathbb{Z})\;,
$$
as topological vector spaces.
\end{corollary}
\begin{proof}
By Prop. \ref{coefficients1} and the isomorphism \eqref{coefficients3} it suffices to show the second equality. By Lemma \ref{projectivelemma}$(ii)$ the spaces coincide as sets. Moreover, it is clear that the seminorms $\sigma_{r_j}$, $r_j \in \mathcal{R}$, are continuous on $s^{\{M_p\}} (\mathbb{Z})$. Conversely, let us prove that every seminorm $p$ on  $s^{\{M_p\}} (\mathbb{Z})$  is continuous on $\varprojlim_{r_j} s^{M_p,r_j}(\mathbb{Z})$. Since $s^{\{M_p\}} (\mathbb{Z})$ is reflexive, there is a bounded set $B$ in  $s'^{\{M_p\}} (\mathbb{Z})$ such that $p$ is bounded by the seminorm
$$
p^B((c_k)_k) = \sup_{(b_k)_k \in B} \left| \sum_{k \in \mathbb{Z}} b_{-k}c_k \right|, \qquad (c_k)_k \in \ s^{\{M_p\}} (\mathbb{Z})\;.
$$
Lemma \ref{projectivelemma}$(i)$ implies that there is $r_j \in \mathcal{R}$ such that
$$
\sup_{(b_k)_k \in B} |b_k| \leq C\E^{M_{r_j}(k)}\;, \qquad k \in \mathbb{Z}\;,
$$
for some $C > 0$. Let $r'_j \in \mathcal{R}$ satisfy the conditions from Lemma \ref{M2sequence}. Hence, Lemma \ref{assumptionM2} yields that
$$
p^B((c_k)_k) \leq C \sum_{k \in \mathbb{Z}}\E^{M_{r'_j}(k)}|c_k|\leq AC\sigma_{r''_j}( (c_k)_k)\sum_{k \in \mathbb{Z}}\E^{-M_{r'_j}(k)} , \qquad (c_k)_k \in \ s^{\{M_p\}} (\mathbb{Z})\;,
$$
where $r''_j = r'_j /(2H)$, $j \in \mathbb{N}$, and $A$,$H$ are the constants occurring in  $(M.2)$.\qed
\end{proof}
\begin{theorem}\label{structure}
Let $N_p$ be a weight sequence satisfying $(M.1)$ and $M_p \prec N_p$. For every $f \in \mathcal{E}_{2\pi}'^{(M_p)}$ ($f \in \mathcal{E}_{2\pi}'^{\{M_p\}}$) there is an 
ultradifferentiable operator $P(D)$ of class $(M_p)$ (of class $\{M_p\}$) and $g \in \mathcal{E}_{2\pi}^{(N_p)}$ such that $f = P(D)g$. In the Beurling case, one can even choose $g \in \mathcal{E}_{2\pi}^{\{M_p\}}$.
\end{theorem}
\begin{proof}
By Prop. \ref{coefficients1} and \ref{coefficients2} it suffices to show that for every $(c_k)_k \in s'^{(M_p)}(\mathbb{Z})$ ($(c_k)_k \in s'^{\{M_p\}}(\mathbb{Z})$) there is 
an ultrapolynomial $P$ of class $(M_p)$ (of class $\{M_p\}$) and $(a_k)_k \in s^{(N_p)}(\mathbb{Z})$ such that $c_k = P(k)a_k$ for all $k 
\in \mathbb{Z}$. In the Beurling case, we must show that it is possible to choose $(a_k)_k \in s^{\{M_p\}}(\mathbb{Z})$. Let $H$ be the constant occuring in $(M.2)$.

\emph{Beurling case}: There is $\lambda > 0$ such that $|c_k| \leq C\E^{M(\lambda k)}$ for all $k \in \mathbb{Z}$ and some $C>0$.
Define
$$
P(z) = \sum_{p = 0}^\infty \frac{(\lambda H^2z)^{2p}}{M_{2p}}\;, \qquad z \in \mathbb{C}\;.
$$
It is clear that $P$ is an ultrapolynomial of class $(M_p)$. Condition $(M.2)$ implies that $P(x) \geq C'\E^{2M(\lambda x)}$ for all $x \in \mathbb{R}$ and some $C'>0$. The ultrapolynomial $P$ and the sequence $a_k = 
c_k/P(k)$ satisfy the requirements.

\emph{Roumieu case}: Lemma \ref{projectivelemma}$(i)$ implies that there is $r_j \in \mathcal{R}$ such that $|c_k| \leq C\E^{M_{r_j}(
k)}$ for all $k \in \mathbb{Z}$ and some $C>0$. On the other hand, by \cite[Lemma 3.4]{Komatsu3} there is $k_j \in \mathcal{R}$ such that $M_p
\prod_{j=0}^pk_j \prec N_p$.
Let $r'_j, k'_j \in \mathcal{R}$ satisfy the conditions from Lemma \ref{M2sequence} with respect to $r_j$ and $k_j$, respectively.
Define
$$
P_1(z) = \sum_{p = 0}^\infty \frac{(2Hz)^{2p}}{\prod_{j=0}^{2p} r'_j M_{2p}}\;, \quad P_2(z) = \sum_{p = 0}^\infty \frac{(2Hz)^{2p}}
{\prod_{j=0}^{2p} k'_j M_{2p}}\;,  \qquad z \in \mathbb{C}\;.
$$
Set $P(z) = P_1(z)P_2(z)$. It is clear that $P$ is an ultrapolynomial of class $\{M_p\}$. Condition $(M.2)$ implies that $P(x) \geq C'\E^{M_{r'_j}(x) + M_{k'_j}(x)}$ for all  $x \in \mathbb{R}$ and some $C'>0$. The ultrapolynomial $P$ and the sequence $a_k = c_k/P(k)$ satisfy the requirements. \qed
\end{proof}
\section{Impossibility Result on the Multiplication \newline of Periodic Ultradistributions}\label{section-impossibility}
In this section we show an analogue of Schwartz' famous impossibility result \cite{Schwartz} for periodic ultradistributions. Recall 
that in the case of classical distributions it states that it is impossible to embed the space of distributions into an associative and commutative differential algebra in such a way that both differentiation, the unity function ($=$ constant function $1$) and the pointwise multiplication of continuous functions is preserved --- or more generally, $k$-times differentiable
functions for any $k \in \mathbb{N}$, see \cite[p.\ 7]{GGK}. In our impossibility result the role of the continuous functions is played
by a class of ultradifferentiable functions which is less regular than the class $\mathcal{E}_{2\pi}^{\ast}$. We assume in the rest of this section that $N_p$ is a weight sequence satisfying $(M.1)$. 
Recall that $\ast$ stands for $(M_p)$ or $\{M_p\}$. In addition, we write $\dagger$ for  $(N_p)$ or $\{N_p\}$. When embedding $\mathcal{E}_{2\pi}'^\ast$ into some associative and commutative algebra $
(\Lambda^{\ast,\dagger}, +, \circ) = \Lambda$ the following requirements seem to be natural:
\begin{enumerate}
\item[]$(i)$\, $\mathcal{E}_{2\pi}'^\ast$ is linearly embedded into $\Lambda$ and $f(x) \equiv 1$ is the unity in $\Lambda$.
\item[]$(ii)$\, For each ultradifferential operator $P(D)$ of class $\ast$ there is a linear operator $P(D):  \Lambda \rightarrow \Lambda
$ satisfying (cf.\ \eqref{Leibniz})
$$
P(D)(\E^{\imag kt} \circ f) = \E^{\imag kt} \circ P(D + k)f\;, \qquad f \in \Lambda\;, k \in \mathbb{Z}\;.
$$
Moreover, $P(D)_{|\mathcal{E}_{2\pi}'^\ast}$ coincides with 
the usual action of $P(D)$ on periodic ultradistributions of class $\ast$.
\item[]$(iii)$\, $\circ_{|\mathcal{E}_{2\pi}^\dagger \times \mathcal{E}_{2\pi}^\dagger}$ coincides with the pointwise product of functions.
\end{enumerate}
\begin{theorem} \label{impossibility}
Suppose that $M_p \prec N_p$. There are no associative and commutative algebras $\Lambda^{(M_p),(N_p)}$  and $\Lambda^{\{M_p\},(N_p)}$  
satisfying conditions $(i)$--$(iii)$. In the Beurling case, there is even no associative and commutative algebra $\Lambda^{(M_p),\{M_p\}}$  
satisfying conditions $(i)$--$(iii)$. 
\end{theorem}
\begin{proof} 
Suppose $\Lambda^{\ast,\dagger} = \Lambda$ is such an algebra (we treat all cases at once). One can generalize Schwartz' original idea by making 
use of the the following observation:  $\E^{\imag kt} \circ f =  \E^{\imag kt} f$ for all $f \in \mathcal{E}_{2\pi}'^\ast$ and $k \in \mathbb{Z}
$; this readily follows from Thm. \ref{structure} and conditions $(ii)$ and $(iii)$.
Define 
$$
\widetilde{\cot}(t) = i +2i \sum^\infty_{k=1}\E^{-2\imag kt}\;,
$$
it is a regularization of the cotangent function in the space of periodic (ultra)distri-butions. The following relations clearly hold in $
\mathcal{E}_{2\pi}'^\ast$
$$
\cos\delta = \delta\;, \qquad \sin\delta  = 0\;, \qquad \cos = \sin \widetilde{\cot}\;.
$$
Since the sine and cosine are finite linear combinations of $\E^{\imag t}$ and $\E^{-\imag t}$, the remark at the beginning of the 
proof implies that
$$\delta = \cos\delta  =\cos \circ \delta =  (\sin \widetilde{\cot}) \circ \delta = (\sin \circ \widetilde{\cot}) \circ \delta = (\sin\delta) \circ 
\widetilde{\cot} = 0\;,
$$ 
contradicting the injectivity of $\mathcal{E}_{2\pi}'^\ast \to \Lambda$. \qed
\end{proof}
Despite this impossibility result, we shall construct algebras $\mathcal{G}_{2\pi}^{(M_p)}$ and $\mathcal{G}_{2\pi}^{\{M_p\}}$ that do satisfy properties $(i)$--$(iii)$ for $\ast = \dagger = (M_p)$ and $\ast = \dagger = \{M_p\}$, respectively.
\section{Algebras of Periodic Generalized Functions  \newline of Class $(M_p)$ and $\{M_p\}$}\label{section-algebra}
We now introduce algebras of periodic generalized functions as quotients of algebras consisting of sequences of ultradifferentiable functions satisfying appropriate growth conditions. Their construction is inspired by the theory of sequence space algebras \cite{Delcroix} (see also \cite{Delcroix2, Valmorin}). In this section we present some basic properties of these algebras. More concretely, we show a null characterization of the space of so called negligible elements, provide an alternative projective description of the algebra of Roumieu type (cf. Cor. \ref{projectiveiso}), define associated rings of generalized numbers and obtain a pointwise characterization of our generalized functions. The embedding of periodic ultradistributions is postponed to the next section. 

\subsection{Definition and Basic Properties} The spaces of \emph{moderate sequences of class} $\ast$  are defined as 
\begin{align*}
\mathcal{E}_{2\pi,\mathcal{M}}^{(M_p)} = \{ (f_n)_n \in {\mathcal{E}_{2\pi}^{(M_p)}}^{\mathbb{N}} \, | \,  &\forall h > 0,\exists \lambda > 0 \\
&\sup_{n \in \mathbb{N}}\|f_n\|_h \E^{-M(\lambda n)} < \infty\}\;,
\end{align*}
\begin{align*}
\mathcal{E}_{2\pi,\mathcal{M}}^{\{M_p\}} = \{ (f_n)_n \in {\mathcal{E}_{2\pi}^{\{M_p\}}}^{\mathbb{N}} \, | \,  &\forall \lambda > 0,\exists h > 0\\ 
&\sup_{n \in \mathbb{N}}\|f_n\|_h \E^{-M(\lambda n)} < \infty\}\;,
\end{align*}
and the spaces of \emph{negligible sequences of class} $\ast$ as
\begin{align*}
\mathcal{E}_{2\pi, \mathcal{N}}^{(M_p)} = \{ (f_n)_n \in {\mathcal{E}_{2\pi}^{(M_p)}}^{\mathbb{N}} \, | \,  &\forall h > 0,\forall \lambda > 0 \\
&\sup_{n \in \mathbb{N}}\|f_n\|_h \E^{M(\lambda n)} < \infty\}\;,
\end{align*}
\begin{align*}
\mathcal{E}_{2\pi, \mathcal{N}}^{\{M_p\}} = \{ (f_n)_n \in {\mathcal{E}_{2\pi}^{\{M_p\}}}^{\mathbb{N}} \, | \,  &\exists \lambda > 0,\exists h > 0\\ 
&\sup_{n \in \mathbb{N}}\|f_n\|_h \E^{M(\lambda n)} < \infty\}\;.
\end{align*}\begin{remark}\label{def-Roumieu}
In the Beurling case our definition fits into the general framework of \cite{Delcroix} while this is not so for the Roumieu case --- the difference lies in the fact that the choice  and order of quantifiers in our definition is completely different. It is important to point out that this will play an essential role when embedding the space of periodic ultradistributions and preserving the product of periodic ultradifferentiable functions (see Sect. \ref{section-embedding}). In particular, for the case of hyperfunctions on the unit circle our definition differs from the one given in \cite{Delcroix2, Delcroix, Valmorin}.
\end{remark}
 Lemma \ref{assumptionM2} ensures that $\mathcal{E}_{2\pi,\mathcal{M}}^\ast$ is an algebra under pointwise operations of sequences and that $\mathcal{E}_{2\pi,\mathcal{N}}^\ast$ is an ideal of it. Hence, we can define the algebra $\mathcal{G}_{2\pi}^\ast$ of \emph{periodic generalized functions of class $\ast$} as the factor algebra
$$
\mathcal{G}_{2\pi}^\ast = \mathcal{E}_{2\pi,\mathcal{M}}^\ast/\mathcal{E}_{2\pi,\mathcal{N}}^\ast\;. 
$$
We denote the equivalence class of $(f_n)_n \in  \mathcal{E}_{2\pi,\mathcal{M}}^\ast$ by $[(f_n)_n]$. Observe that $\mathcal{E}_{2\pi}^\ast$ can be regarded as a subalgebra of $\mathcal{G}_{2\pi}^\ast$ via the constant embedding
$$
\sigma(f) := [(f)_n]\;, \qquad f \in  \mathcal{E}_{2\pi}^\ast\;.
 $$
We also remark that $\mathcal{G}_{2\pi}^\ast$ can be endowed with a canonical action of ultradifferential operators of class $\ast$: The spaces $\mathcal{E}_{2\pi,\mathcal{M}}^\ast$ and $\mathcal{E}_{2\pi,\mathcal{N}}^\ast$ are closed under ultradifferential operators $P(D)$ of class $\ast$ if we define their actions on sequences as $P(D)((f_n)_n) := (P(D)f_n)_n$. Consequently, every ultradifferential operator $P(D)$ of class $\ast$ induces a linear operator
$$
P (D) :  \mathcal{G}_{2\pi}^\ast \to \mathcal{G}_{2\pi}^\ast\;,
$$
which clearly satisfies the generalized Leibniz' rule \eqref{Leibniz} for any $f \in \mathcal{G}_{2\pi}^\ast$ and $\E^{\imag kt}$, $k \in \mathbb{Z}$, identified with $\sigma(\E^{\imag kt})$. 

We can also define spaces of moderate and negligible sequences  based on the spaces $s^\ast(\mathbb{Z})$
\begin{align*}
s_\mathcal{M}^{(M_p)}(\mathbb{Z}) = \{ (c_{k,n})_{k,n} \in {s^{(M_p)}(\mathbb{Z})}^{\mathbb{N}} \, | \,  &\forall h > 0,\exists \lambda > 0\\ 
&\sup_{n \in \mathbb{N}}\sigma_h((c_{k,n})_k) \E^{-M(\lambda n)} < \infty \}\;,
\end{align*}
\begin{align*}
s_\mathcal{M}^{\{M_p\}}(\mathbb{Z}) = \{ (c_{k,n})_{k,n} \in  {s^{\{M_p\}}(\mathbb{Z})}^{\mathbb{N}} \, | \,  &\forall \lambda > 0,\exists h > 0\\ 
&\sup_{n \in \mathbb{N}}\sigma_h((c_{k,n})_k) \E^{-M(\lambda n)} < \infty\}\;,
\end{align*}
and  
\begin{align*}
s_\mathcal{N}^{(M_p)}(\mathbb{Z}) = \{ (c_{k,n})_{k,n} \in {s^{(M_p)}(\mathbb{Z})}^{\mathbb{N}} \, | \,  &\forall h > 0,\forall \lambda > 0\\ 
&\sup_{n \in \mathbb{N}}\sigma_h((c_{k,n})_k) \E^{M(\lambda n)} < \infty \}\;,
\end{align*}
\begin{align*}
s_\mathcal{N}^{\{M_p\}}(\mathbb{Z}) = \{ (c_{k,n})_{k,n} \in  {s^{\{M_p\}}(\mathbb{Z})}^{\mathbb{N}} \, | \,  &\exists \lambda > 0,\exists h > 0\\ 
&\sup_{n \in \mathbb{N}}\sigma_h((c_{k,n})_k) \E^{M(\lambda n)} < \infty\}\;.
\end{align*}
Proposition \ref{coefficients1} implies the following simple but useful lemma:
\begin{lemma}\label{coefficients4}
Let $(f_n)_n \in {\mathcal{E}_{2\pi}^\ast}^{\mathbb{N}}$. Then 
\begin{enumerate}
\item[]$(i)$\, $(f_n)_n  \in \mathcal{E}_{2\pi,\mathcal{M}}^\ast$ if and only if $(\widehat{f}_n(k))_{k,n} \in s_\mathcal{M}^\ast(\mathbb{Z})$. 
\item[]$(ii)$\, $(f_n)_n  \in \mathcal{E}_{2\pi,\mathcal{N}}^\ast$ if and only if $(\widehat{f}_n(k))_{k,n} \in s_\mathcal{N}^\ast(\mathbb{Z})$. 
\end{enumerate}
\end{lemma}
We now show the null characterization of the ideal $\mathcal{E}_{2\pi,\mathcal{N}}^\ast$. 
\begin{proposition} \label{nullchar}
Let $(f_n)_n \in \mathcal{E}_{2\pi,\mathcal{M}}^\ast$. Then, $(f_n)_n \in \mathcal{E}_{2\pi,\mathcal{N}}^\ast$ if and only if 
$$
\sup_{n \in \mathbb{N}} \|f_n\|_{L^\infty(\mathbb{R})}\E^{M(\lambda n)} < \infty\;. 
$$ 
for all $\lambda > 0$ (for some $\lambda > 0$).
\end{proposition}
\begin{proof}
The proof is based on the Landau-Kolmogorov inequality \cite{Kolmogorov} on the real line: For all $0 < p < m \in \mathbb{N}$ we have
\begin{equation}
\|D^pf\|_{L^\infty(\mathbb{R})} \leq 2\pi \|f\|_{L^\infty(\mathbb{R})}^{1-p/m} \|D^mf\|_{L^\infty(\mathbb{R})}^{p/m}\;, \qquad f \in C^m(\mathbb{R})\;,
\label{Landau}
\end{equation}
provided $D^nf \in L^\infty(\mathbb{R})$ for $n= 0, \ldots, m$. For every $h > 0$ there is $\lambda > 0$ (for every $\lambda > 0$ there is $h > 0$) such that
$$
\|D^pf_n\|_{L^\infty(\mathbb{R})} \leq \frac{C\E^{M(\lambda n)} M_p}{h^p}\;, \qquad p,n \in \mathbb{N}\;,
$$
for some $C = C_{\lambda,h} > 0$. The assumption yields that for every $\mu > 0$ (some $\mu >0$) 
$$
\|f_n\|_{L^\infty(\mathbb{R})} \leq C'\E^{-M(\mu n)}\;, \qquad n \in \mathbb{N}\;,
$$
for some $C' = C'_{\mu} > 0$. By applying inequality \eqref{Landau} with $m = 2p$, $p > 0$, we obtain that 
\begin{align*}
\|D^pf_n\|_{L^\infty(\mathbb{R})} &\leq 2\pi \|f_n\|_{L^\infty(\mathbb{R})}^{1/2} \|D^{2p}f\|_{L^\infty(\mathbb{R})}^{1/2}\\
&\leq \frac{2\pi (CC')^{1/2}M^{1/2}_{2p}}{h^p}\exp\left(\frac{M(\lambda n) - M(\mu n)}{2}\right), \qquad p,n \in \mathbb{N}.
\end{align*}
The result follows from Lemma \ref{assumptionM2} and the fact that $(M.2)$ implies that $M_{2p}^{1/2} \leq A^{1/2}H^p M_p$, $p \in \mathbb{N}$. \qed
\end{proof}
\begin{corollary}\label{nullcharseq}
Let $(c_{k,n})_{k,n} \in s_\mathcal{M}^\ast(\mathbb{Z})$. Then, $(c_{k,n})_{k,n}\in s_\mathcal{N}^\ast(\mathbb{Z})$ if and only if 
$$
\sup_{n \in \mathbb{N}}\sup_{k \in \mathbb{Z}} |c_{k,n}|\E^{M(\lambda n)} < \infty\;, 
$$ 
for all $\lambda > 0$ (for some $\lambda > 0$).
\end{corollary}
\subsection{Projective Description of $\mathcal{G}_{2\pi}^{\{M_p\}}$}\label{subsection-projective}
 We give an alternative projective description of the algebra of periodic generalized functions of Roumieu type. As in Sect. \ref{section-structure} this shall be done by using the family $\mathcal{R}$ of all non-decreasing sequences tending to infinity. We define projective type spaces of moderate and negligible sequences by means of the seminorms \eqref{seminorms}
 as
\begin{align*}
\tilde{\mathcal{E}}^{\{M_p\}}_{2\pi,\mathcal{M}}= \{ (f_n)_n \in {\mathcal{E}_{2\pi}^{\{M_p\}}}^{\mathbb{N}} \, | \, &\forall r_j \in \mathcal{R},\exists s_j \in \mathcal{R} \\
&\sup_{n \in \mathbb{N}}\|f_n\|_{r_j} \E^{-M_{s_j}(n)} < \infty \}\;,
\end{align*}
and
\begin{align*}
\tilde{\mathcal{E}}^{\{M_p\}}_{2\pi,\mathcal{N}}= \{ (f_n)_n \in {\mathcal{E}_{2\pi}^{\{M_p\}}}^{\mathbb{N}} \, | \, &\forall r_j \in \mathcal{R},\forall s_j \in \mathcal{R} \\
&\sup_{n \in \mathbb{N}}\|f_n\|_{r_j} \E^{M_{s_j}(n)} < \infty \}\;.
\end{align*}
Likewise, we define 
\begin{align*}
\tilde{s}_\mathcal{M}^{\{M_p\}}(\mathbb{Z}) = \{ (c_{k,n})_{k,n} \in  {s^{\{M_p\}}(\mathbb{Z})}^{\mathbb{N}} \, | \, &\forall r_j \in \mathcal{R}, \exists s_j \in \mathcal{R}\\
&\sup_{n \in \mathbb{N}}\sigma_{r_j}((c_{k,n})_k) \E^{-M_{s_j}(n)} < \infty \} \;,
\end{align*}
and
\begin{align*}
\tilde{s}_\mathcal{N}^{\{M_p\}}(\mathbb{Z}) = \{ (c_{k,n})_{k,n} \in  {s^{\{M_p\}}(\mathbb{Z})}^{\mathbb{N}} \, | \, &\forall r_j \in \mathcal{R}, \forall s_j \in \mathcal{R}\\
&\sup_{n \in \mathbb{N}}\sigma_{r_j}((c_{k,n})_k) \E^{M_{s_j}(n)} < \infty \} \;.
\end{align*}
The goal of this subsection is to show the following result:
\begin{proposition}\label{projective}
We have 
$$
\mathcal{E}^{\{M_p\}}_{2\pi,\mathcal{M}} = \tilde{\mathcal{E}}^{\{M_p\}}_{2\pi,\mathcal{M}}\;, \quad \mathcal{E}^{\{M_p\}}_{2\pi,\mathcal{N}} =  \tilde{\mathcal{E}}^{\{M_p\}}_{2\pi,\mathcal{N}}\;,
$$
and
$$
s^{\{M_p\}}_{\mathcal{M}}(\mathbb{Z}) = \tilde{s}^{\{M_p\}}_{\mathcal{M}}(\mathbb{Z})\;, \quad s^{\{M_p\}}_{\mathcal{N}}(\mathbb{Z}) =  \tilde{s}^{\{M_p\}}_{\mathcal{N}}(\mathbb{Z})\;.
$$
\end{proposition}
\begin{remark}\label{justification}
Since the structure (choice and order of quantifiers) of $\tilde{\mathcal{E}}^{\{M_p\}}_{2\pi,\mathcal{M}}$ and $\tilde{\mathcal{E}}^{\{M_p\}}_{2\pi, \mathcal{N}}$ coincides with the structure of the widely accepted definition for spaces of moderate and negligible sequences based on an arbitrary locally convex space \cite{Delcroix}, Prop. \ref{projective} may serve as a justification for our definitions in the Roumieu case.
\end{remark}
\begin{proof}
The isomorphism \eqref{coefficients3} yields the following: For $(f_n)_n \in {\mathcal{E}_{2\pi}^{\{M_p\}}}^\mathbb{N}$ we have $(f_n)_n \in \tilde{\mathcal{E}}^{\{M_p\}}_{2\pi, \mathcal{M}}$ if and only if $(\widehat{f}_n(k))_{k,n} \in \tilde{s}^{\{M_p\}}_{\mathcal{M}}(\mathbb{Z})$. A similar statement holds for the null ideals. By Lemma \ref{coefficients4} it therefore suffices to show the second pair of equalities. We start by proving $s^{\{M_p\}}_{\mathcal{M}}(\mathbb{Z}) =  \tilde{s}^{\{M_p\}}_{\mathcal{M}}(\mathbb{Z})$. Let $(c_{k,n})_{k,n} \in s^{\{M_p\}}_{\mathcal{M}}(\mathbb{Z})$ and fix $r_j \in \mathcal{R}$. Since the seminorm $\sigma_{r_j}$ is continuous on $s^{M_p,h} (\mathbb{Z})$ for each $h >0$, we obtain  $\sup_n\sigma_{r_j}((c_{k,n})_k) \E^{-M(\lambda n)} < \infty$ for all $\lambda > 0$. By applying Lemma \ref{projectivelemma}$(i)$ to the sequence $(\sigma_{r_j}((c_{k,n})_k))_n$, we find a sequence $s_j \in \mathcal{R}$ such that $\sup_n \sigma_{r_j}((c_{k,n})_k) \E^{-M_{s_j}(n)} < \infty$. Conversely, let $(c_{k,n})_{k,n} \in \tilde{s}^{\{M_p\}}_{\mathcal{M}}(\mathbb{Z})$. The definition of $\tilde{s}^{\{M_p\}}_{\mathcal{M}}(\mathbb{Z})$ and Lemma \ref{projectivelemma}$(i)$ imply that for all $\lambda > 0$ and $r_j \in \mathcal{R}$ 
$$
\sup_{n \in \mathbb{N}}\sup_{k \in \mathbb{Z}} |c_{k,n}|\exp\left(M_{r_j}(k)-M(\lambda n)\right) < \infty\;. 
$$
Fix $\lambda > 0$. By applying Lemma \ref{projectivelemma}$(ii)$ to the sequence 
$$
\left(\max_{|k|=l}\sup_{n \in \mathbb{N}}|c_{k,n}|\E^{-M(\lambda n)}\right)_{l \in \mathbb{N}}\;,
$$ 
we find $h>0$ such that $\sup_n\sigma_h((c_{k,n})_k) \E^{-M(\lambda n)} < \infty$. For the second equality, the inclusion $s^{\{M_p\}}_{\mathcal{N}}(\mathbb{Z})  \subseteq  \tilde{s}^{\{M_p\}}_{\mathcal{N}}(\mathbb{Z})$ is clear, whereas the converse inclusion is a consequence of Cor. \ref{nullcharseq} and  Lemma \ref{projectivelemma}$(ii)$.\qed
\end{proof}
\subsection{Generalized Point Values}
In this subsection we introduce the ring of generalized numbers of class $\ast$ in order to view periodic generalized functions of class $\ast$ as objects defined pointwise. We introduce
$$
 \mathbb{C}_\mathcal{M}^{(M_p)} = \{ (z_n)_n \in \mathbb{C}^{\mathbb{N}} \, | \,  \exists \lambda > 0 \sup_{n \in \mathbb{N}}|z_n|\E^{-M(\lambda n)} < \infty\}\;,
$$
$$
 \mathbb{C}_\mathcal{M}^{\{M_p\}} = \{ (z_n)_n \in \mathbb{C}^{\mathbb{N}} \, | \,  \forall \lambda > 0 \sup_{n \in \mathbb{N}}|z_n|\E^{-M(\lambda n)} < \infty\}\;,
 $$
and
$$
\mathbb{C}_\mathcal{N}^{(M_p)} = \{ (z_n)_n \in \mathbb{C}^{\mathbb{N}} \, | \,  \forall \lambda > 0 \sup_{n \in \mathbb{N}}|z_n|\E^{M(\lambda n)} < \infty\}\;,
$$
$$
\mathbb{C}_\mathcal{N}^{\{M_p\}} =\{ (z_n)_n \in \mathbb{C}^{\mathbb{N}} \, | \,  \exists \lambda > 0 \sup_{n \in \mathbb{N}}|z_n|\E^{M(\lambda n)} < \infty\}\;.
$$
Clearly $\mathbb{C}_\mathcal{N}^{\ast}$ is an ideal in the ring $\mathbb{C}_\mathcal{M}^{\ast}$. The ring of \emph{generalized numbers of class $\ast$} is defined as the factor ring
$$ 
\widetilde{\mathbb{C}}^\ast = \mathbb{C}_\mathcal{M}^{\ast}/\mathbb{C}_\mathcal{N}^{\ast}. 
$$
Observe that $\widetilde{\mathbb{C}}^\ast$ is not a field. In fact, this follows from  the same examples used for the ring of Colombeau generalized numbers \cite[Ex.\ 1.2.33, p.\ 32]{GGK}. Furthermore, the elements of $\mathbb{C}$ are canonically embedded into $\widetilde{\mathbb{C}}^\ast$ via the constant embedding
$$
\sigma(z) := [(z)_n]\;, \qquad z \in  \mathbb{C}\;.
 $$
Likewise, one can define the subring $\widetilde{\mathbb{R}}^\ast$. Let $f = [(f_n)_n] \in \mathcal{G}_{2\pi}^\ast$ and $t = [(t_n)_n] \in \widetilde{\mathbb{R}}^\ast$. The point value of $f$ at $t$ is defined as $f(t) := [(f_n(t_n))_n]$. The point value $f(t)$ does not depend on the representative of $f$ nor on the representative of $t$; the former is clear while the latter follows from the mean value theorem.The induced pointwise defined mapping $f:\widetilde{\mathbb{R}}^\ast \to \widetilde{\mathbb{C}}^\ast: t \to f(t)$ is $2\pi$-periodic, that is,
$$
f(t+ \sigma(2\pi)) = f(t)\;, \qquad t \in \widetilde{\mathbb{R}}^\ast\;.
$$
The next proposition shows that every $f \in  \mathcal{G}_{2\pi}^\ast$ can be associated with the mapping $f:\widetilde{\mathbb{R}}^\ast \to \widetilde{\mathbb{C}}^\ast$ in a one-to-one fashion. We define 
$$
\widetilde{[0,2\pi]}^\ast = \{ t \in \widetilde{\mathbb{R}}^\ast \, | \, \exists (t_n)_n \in \mathbb{C}^\ast_{\mathcal{M}} \mbox{ such that } t = [(t_n)_n] \mbox{ and } t_n \in [0,2\pi], n \in \mathbb{N} \}.
$$
\begin{proposition}
Let $f \in  \mathcal{G}_{2\pi}^\ast$. Then, $f = 0$ in $\mathcal{G}_{2\pi}^\ast$ if and only if $f(t) = 0$ in $\widetilde{\mathbb{C}}^\ast$ for all $t \in \widetilde{[0,2\pi]}^\ast$.
\end{proposition}
\begin{proof}
The direct implication is clear. Conversely, suppose that $f = [(f_n)_n] \neq 0$ in $\mathcal{G}_{2\pi}^\ast$. By Prop. \ref{nullchar} there are $\lambda > 0$, $m_n \nearrow \infty$ and $t_n \in [0,2\pi]$ ($m_n \nearrow \infty$ and $t_n \in [0,2\pi]$) such that
$$
|f_{m_n}(t_n)| \geq n\E^{-M(\lambda m_n)} \qquad ( |f_{m_n}(t_n)| \geq n \E^{-M(m_n/n)})\;, \qquad  n \in \mathbb{N}\;. 
$$
Define $t'_l = t_{n}$ if $l = m_n$ for some $n \in \mathbb{N}$ and as $0$ otherwise. For $t' = [(t'_l)_l] \in \widetilde{[0,2\pi]}^\ast$ we have $f(t') \neq 0$ in $\mathbb{C}^\ast$.\qed
\end{proof}
\section{Embedding of Periodic Ultradistributions} \label{section-embedding}
In this section we embed the space of periodic ultradistributions of class $\ast$ into the algebra $\mathcal{G}_{2\pi}^\ast$ in such a way that the multiplication of periodic ultradifferentiable functions of class $\ast$ is preserved. As usual, we accomplish this by means of convolution with a suitable mollifier sequence. Let us start with defining the type of mollifier sequences that will be employed in our embedding. 
\begin{definition}\label{mollifier}
A sequence $(\varphi_n)_{n \in \mathbb{N}}$ of smooth $2\pi$-periodic functions is called a \emph{mollifier sequence} if the Fourier coefficients $\widehat{\varphi}_n(k) =c_{k,n}$, $k \in \mathbb{Z}$, $n \in \mathbb{N}$, satisfy the following conditions: 
\begin{itemize}
\item There is $C>0$ such that 
$$
|c_{k,n}| \leq C\;, \qquad k \in \mathbb{Z}\;, n \in \mathbb{N}\;.
$$
\item There is $R >0$ such that for all $n \in \mathbb{N}$
$$
c_{k,n} = 0\;, \qquad |k| \geq Rn\;.
$$

\item There is $r >0$ such that for all $n \in \mathbb{N}$
$$
c_{k,n} = \frac{1}{2\pi}\;, \qquad |k| \leq rn\;.
$$
\end{itemize}
\end{definition}
\begin{example} Define 
$$
\varphi_n(t) = \frac{1}{2\pi}\sum_{|k| \leq n} \E^{\imag kt}\;, \qquad n \in \mathbb{N}\;,
$$
then $(\varphi_n)_n$ is clearly a mollifier sequence. It is used in \cite{Valmorin} to embed the space of periodic hyperfunctions into some Colombeau type algebra.
 \end{example}

\begin{example} Let $\psi$ be a compactly supported continuous function on $\mathbb{R}$ such that $\psi \equiv 1/(2\pi)$ in a neighbourhood of the origin. Let $\psi_n = \psi (\cdot / n)$ and define
$$
\varphi_n(t) = \sum_{k \in \mathbb{Z}}\int_\mathbb{R} \psi_n(\xi)\E^{\imag \xi(t+ 2\pi k)}\: \D \xi\;, \qquad n \in \mathbb{N}\;,
$$
then $\varphi_n$ is a smooth $2\pi$-periodic function with $\widehat{\varphi}_n(k) = \psi(k/n)$ for all $k \in \mathbb{Z}$, hence $(\varphi_n)_n$ is a mollifier sequence. These kind of mollifier sequences are closely related to the mollifiers usually used to embed the space of compactly supported (ultra)distributions into some Colombeau type algebra (see the discussions at the beginning of \cite[Sect.\ 1.2.2]{GGK} and \cite[Sect.\ 5]{DVV}). 
\end{example}
\begin{theorem}\label{main}
Let $(\varphi_n)_n$ be a mollifier sequence. Then 
$$
\iota: \mathcal{E}_{2\pi}'^\ast \to \mathcal{G}_{2\pi}^\ast: f \rightarrow [(f \ast \varphi_n)_n]
$$
is a linear embedding that satisfies the following properties:
\begin{enumerate}
\item[] $(i)$\, $\iota$ commutes with ultradifferential operators of class $\ast$, that is, for all ultra-differential operators $P(D)$ of class $\ast$
$$
P(D)\iota(f) = \iota(P(D)f)\;, \qquad f \in \mathcal{G}_{2\pi}^\ast\;.
$$
\item[] $(ii)$\, $\iota_{|\mathcal{E}_{2\pi}^\ast}$ coincides with the constant embedding  $\sigma$. Consequently, 
$$
\iota(fg) =\iota(f)\iota(g)\;, \qquad f,g \in \mathcal{E}_{2\pi}^\ast\;.
$$
\end{enumerate}
\end{theorem}
\begin{proof}
We use the same notation as in Def. \ref{mollifier}. Let $A,H$ be the constants occurring in $(M.2)$. We start by proving that $(f \ast \varphi_n)_n \in \mathcal{E}_{2\pi,\mathcal{M}}^\ast$. By Lemma \ref{coefficients4} it suffices to show that  $(\widehat{f}(k)c_{k,n})_{k,n} \in s^\ast_\mathcal{M}(\mathbb{Z})$. Proposition \ref{coefficients2} implies that for some $\lambda >0$ (every $\lambda >0$) $K = \sigma'_\lambda((\widehat{f}(k))_k) < \infty$. Hence for $h>0$ it holds that
\begin{align*}
\sigma_h((\widehat{f}(k)c_{k,n})_{k}) &= \sup_{k \in \mathbb{Z}} |\widehat{f}(k)||c_{k,n}|\E^{M(hk)} \\ 
&\leq CK\sup_{|k| \leq Rn}\exp\left(M(\lambda k)+M(hk)\right) \\
&\leq ACK\E^{M(\mu n)}\;, \quad n \in \mathbb{N}\;,
\end{align*}
where $\mu = HR\max(\lambda, h)$. This shows that $(f \ast \varphi_n)_n \in \mathcal{E}_{2\pi,\mathcal{M}}^\ast$ both in the Beurling and Roumieu case. The injectivity of $\iota$ follows from the fact that $\varphi_n \to \delta$ in $\mathcal{E}_{2\pi}'^\ast$ as $n \to \infty$. Property $(i)$ is clear. Finally, we show $(ii)$. Let $f \in \mathcal{E}_{2\pi}^\ast$. By Lemma \ref{coefficients4} and Prop. \ref{nullchar} it suffices to show that  
$$
\sup_{n \in \mathbb{N}}\sup_{k \in \mathbb{Z}}|\widehat{f}(k)| |1-2\pi c_{k,n}|\E^{M(\lambda n)} < \infty\;,
$$
 for all $\lambda > 0$ (for some $\lambda > 0$). Proposition \ref{coefficients1} implies that for every $\lambda >0$ (some $\lambda >0$) $K = \sigma_{\lambda}((\widehat{f}(k))_k) < \infty$. Hence 
\begin{align*}
|\widehat{f}(k)||1-2\pi c_{k,n}| & \leq K|1-2\pi c_{k,n}|\E^{-M(\lambda k)} \\ 
&\leq (1+2\pi C)K\E^{-M(\lambda rn)}\;, \quad n \in \mathbb{N}\;, k \in \mathbb{Z}\;.
\end{align*}
\qed
\end{proof}
\begin{remark}
It is clear that the embedding $\iota: \mathcal{E}_{2\pi}'^\ast \to \mathcal{G}_{2\pi}^\ast$ satisfies the properties $(i)$--$(iii)$ from Sect. \ref{section-impossibility} with $\dagger = \ast$. Hence, it is optimal in the sense discussed there.
\end{remark}
\section{Regular Periodic Generalized Functions \newline of Class $(M_p)$ and $\{M_p\}$}\label{section-regular}
In this section we introduce a notion of regularity (with respect to ultradifferentiability of class $\ast$) in $\mathcal{G}_{2\pi}^{\ast}$. The definition given below is based on the regular algebra $\mathcal{G}^{\infty}$ frequently used in classical Colombeau theory \cite{Ober}. Throughout this section we fix a mollifier sequence $(\varphi_n)_n$ with $r = 1$ in Def. \ref{mollifier} and consider its associated embedding $\iota$. We define the algebra of \emph{regular} periodic generalized functions of class $\ast$ as
$$
\mathcal{G}_{2\pi}^{\ast,\infty} = \mathcal{E}_{2\pi,\mathcal{M}}^{\ast, \infty}/\mathcal{E}_{2\pi,\mathcal{N}}^\ast\;, 
$$
where
\begin{align*}
\mathcal{E}_{2\pi,\mathcal{M}}^{(M_p), \infty} = \{ (f_n)_n \in {\mathcal{E}_{2\pi}^{(M_p)}}^{\mathbb{N}} \, | \,   &\exists \lambda > 0,\forall h > 0 \\
&\sup_{n \in \mathbb{N}}\|f_n\|_h \E^{-M(\lambda n)} < \infty\}\;,
\end{align*}
\begin{align*}
\mathcal{E}_{2\pi,\mathcal{M}}^{\{M_p\}, \infty} = \{ (f_n)_n \in {\mathcal{E}_{2\pi}^{\{M_p\}}}^{\mathbb{N}} \, | \,  &\exists h > 0,\forall \lambda > 0 \\
&\sup_{n \in \mathbb{N}}\|f_n\|_h \E^{-M(\lambda n)} < \infty\}\;.
\end{align*}
\begin{lemma}\label{lemmareg}
Let $f \in \mathcal{E}_{2\pi}'^\ast$ and suppose $(f_n)_n \in \mathcal{E}_{2\pi,\mathcal{M}}^{\ast}$ is a representative of $\iota(f)$. Then,
$$
\sup_{n \in \mathbb{N}}\sigma'_{H\lambda}((\widehat{f}(k)-\widehat{f}_n(k))_k) \E^{M(\lambda  n)} < \infty\;,
$$
for some $\lambda > 0$, where $H$ denotes the constant occurring in $(M.2)$.
\end{lemma}
\begin{proof}
The assertion is independent of the chosen representative and therefore we may assume that $f_n = f \ast \varphi_n$, $n \in \mathbb{N}$. Proposition \ref{coefficients2} implies that $K = \sigma'_{\lambda}((\widehat{f}(k))_k) < \infty$ for some $\lambda >0$. Hence
\begin{align*}
\sigma'_{H\lambda}((\widehat{f}(k)(1-2\pi c_{k,n}))_{k}) &= \sup_{k \in \mathbb{Z}} |\widehat{f}(k)||1-2\pi c_{k,n}|\E^{-M(H\lambda k)} \\ 
&\leq (1+2\pi C)K\sup_{|k| \geq n}\exp \left(M(\lambda k)-M(H\lambda k)\right) \\
&\leq A(1+2\pi C)K\E^{- M(\lambda n)}, \quad n \in \mathbb{N}\;. 
\end{align*}
  \qed
\end{proof}
The next regularity theorem shows that the notion of $\mathcal{G}_{2\pi}^{\ast,\infty}$-regularity coincides with ultradifferentiability of class $\ast$ when restricted to the image of $\mathcal{E}_{2\pi}'^\ast$ under the embedding $\iota$.
\begin{theorem}\label{mainreg} We have
$$
\mathcal{G}_{2\pi}^{\ast,\infty} \cap \iota(\mathcal{E}_{2\pi}'^\ast) =  \iota(\mathcal{E}_{2\pi}^\ast)\;.
$$
\end{theorem}
\begin{proof}
The inclusion $\iota(\mathcal{E}_{2\pi}^\ast) \subseteq \mathcal{G}_{2\pi}^{\ast,\infty} \cap \iota(\mathcal{E}_{2\pi}'^\ast)$ is obvious. Conversely, let $f \in \mathcal{E}_{2\pi}'^\ast$ and assume  $\iota(f) \in \mathcal{G}_{2\pi}^{\ast,\infty}$. By Prop. \ref{coefficients1} it suffices to show that $(\widehat{f}(k))_k \in s^\ast(\mathbb{Z})$. Let $(f_n)_n \in \mathcal{E}_{2\pi,\mathcal{M}}^{\ast}$ be a representative of $\iota(f)$ and set $g_n = f -f_n$, $n \in \mathbb{N}$. From the definition of $\mathcal{G}_{2\pi}^{\ast,\infty}$ it follows that 
$$
\exists \lambda > 0,\forall h > 0 \qquad (\exists h > 0,\forall \lambda > 0)  
$$
$$
\sup_{n \in \mathbb{N}}\sigma_h((\widehat{f}_n(k))_k) \E^{-M(\lambda n)} < \infty\;.
$$
Furthermore, Lemma \ref{lemmareg} implies that for some $l >0$
$$
\sup_{n \in \mathbb{N}}\sigma'_{Hl}((\widehat{g}_n(k))_k) \E^{M(l  n)} < \infty\;.
$$
Combining these two facts we obtain that
$$
\exists \lambda,l > 0,\forall h > 0 \qquad (\exists h,l > 0,\forall \lambda > 0)  
$$
\begin{equation}
|\widehat{f}(k)| \leq K\left(\exp\left(M(\lambda n) - M(hk)\right) + \exp\left(M(Hlk) - M(ln)\right)\right),  \quad k \in \mathbb{Z}\;, n \in \mathbb{N}\;,
\label{mainineq}
\end{equation}
where $K$ is a positive constant independent of $k$ and $n$. In the remainder of the proof we treat the Beurling and Roumieu case separately. For a positive real number $x$ we write $\lceil x \rceil$ for the smallest natural number $n$ such that $n \geq x$.

\emph{Beurling case:} Let $\mu > 0$ be arbitrary. Define  $m = \lceil H\max(Hl,\mu)/l \rceil$. Inequality \eqref{mainineq} with $n = m|k|$ and $h = H\max(\lambda m, \mu)$ implies that
$$
|\widehat{f}(k)| \leq 2AK\E^{-M(\mu k)}\;, \qquad k \in \mathbb{Z}\;.
$$

\emph{Roumieu case:} Define  $m = \lceil H^2 \rceil$. Inequality \eqref{mainineq} with $n = m|k|$ and $\lambda = h/(Hm)$ implies that
$$
|\widehat{f}(k)| \leq 2AK\E^{-M(\mu k)}\;, \qquad k \in \mathbb{Z}\;,
$$
where $\mu = \min(h/H,Hl)$.\qed
\end{proof}

\end{document}